\documentclass[11pt,twoside]{article}
\usepackage{amssymb,amsmath,amsthm,amsfonts,mathrsfs,hyperref}
\usepackage[T1]{fontenc}
\usepackage{times}
\usepackage{enumerate,cite,mathtools,enumitem,microtype,needspace}
\allowdisplaybreaks[2]
\AtBeginDocument{%
  \setlength{\abovedisplayskip}{10pt plus 3pt minus 2pt}%
  \setlength{\belowdisplayskip}{10pt plus 3pt minus 2pt}%
  \setlength{\abovedisplayshortskip}{5pt plus 2pt}%
  \setlength{\belowdisplayshortskip}{7pt plus 2pt minus 2pt}%
}
\hypersetup{colorlinks=true,linkcolor=blue,citecolor=red,urlcolor=cyan,
pdftitle={Properties and applications of Muckenhoupt A1 growth weights},pdfauthor={Dinghuai Wang}}
\newtheorem{theorem}{Theorem}[section]
\newtheorem{lemma}[theorem]{Lemma}
\newtheorem{proposition}[theorem]{Proposition}
\newtheorem{corollary}[theorem]{Corollary}
\theoremstyle{definition}

\theoremstyle{plain}
\numberwithin{equation}{section}
\newcommand{\R}{\mathbb R}
\newcommand{\ind}{\mathbf 1}
\newcommand{\dx}{\,dx}
\newcommand{\BMO}{\mathrm{BMO}}
\title{Properties and applications of Muckenhoupt $A_1$ growth weights}
\author{Dinghuai Wang}
\date{}
\makeatletter
\renewcommand{\@maketitle}{%
  \newpage\null\vskip2em
  \begin{center}
    \let\footnote\thanks
    {\LARGE\@title\par}
    \ifx\@author\@empty\else
      \vskip1.5em
      {\large\lineskip.5em\begin{tabular}[t]{c}\@author\end{tabular}\par}
    \fi
    \ifx\@date\@empty\else\vskip1em{\large\@date\par}\fi
  \end{center}
  \par\vskip1.5em
}
\makeatother
\begin{document}
\maketitle
\thispagestyle{myheadings}

\begin{abstract}
For $\beta\ge0$, let $A_1^{(\beta)}$ be the class of weights
satisfying
$$
 [w]_{A_1^{(\beta)}}:=
 \sup_{0<\varepsilon\le1}
       \varepsilon^\beta[w]_{A_{1+\varepsilon}}<\infty.
$$
These classes satisfy
$A_1\subsetneq A_1^{(\beta)}\subsetneq\bigcap_{p>1}A_p$
for $\beta>0$. For every $\alpha\ge0$, we prove that
$w\in A_1^{(\beta)}$ if and only if
$$
 w\bigl(\{M_{L(\log L)^\alpha}f>\lambda\}\bigr)
 \le C\int_{\R^n}\frac{|f(x)|}{\lambda}
 \left[\log\left(e+\frac{|f(x)|}{\lambda}\right)\right]^{\alpha+\beta}
 w(x)\dx.
$$
The least constant $C$ is comparable to $[w]_{A_1^{(\beta)}}$.
Also, we obtain
corresponding characterizations for Calder\'on--Zygmund
operators and their commutators.

\end{abstract}

\section{Introduction and main results}\label{sec:intro}

Let $w$ be a locally integrable function on $\R^n$ with
$0<w<\infty$ almost everywhere. For cubes $Q$ with sides parallel
to the coordinate axes, write
$$
 w(E)=\int_Ew(x)\dx,\qquad
 w_Q=\frac{1}{|Q|}\int_Qw(x)\dx.
$$
Let $M$ be the uncentered Hardy--Littlewood maximal operator,
$$
 Mf(x)=\sup_{Q\ni x}\frac{1}{|Q|}\int_Q|f(y)|\,dy.
$$
The Muckenhoupt characteristics are
\begin{equation}\label{eq:Ap}
 \begin{aligned}
 [w]_{A_p}
 &=\sup_Q w_Q\left(\frac{1}{|Q|}
          \int_Qw(x)^{-\frac{1}{p-1}}\dx\right)^{p-1},
 &&1<p<\infty,\\
 [w]_{A_1}
 &=\sup_Q\frac{w_Q}{\operatorname*{ess\,inf}_Qw}.
 \end{aligned}
\end{equation}
These classes characterize the weighted estimates for $M$.
Muckenhoupt's theorem \cite{Muckenhoupt} gives
$$
 \begin{aligned}
 w\in A_p
 &\quad\Longleftrightarrow\quad
 \int_{\R^n}(Mf)^p w\dx\le C\int_{\R^n}|f|^p w\dx,
 \qquad1<p<\infty,\\
 w\in A_1
 &\quad\Longleftrightarrow\quad
 \sup_{\lambda>0}\lambda w(\{Mf>\lambda\})
       \le C\int_{\R^n}|f|w\dx,
 \end{aligned}
$$
where the estimates hold for all $f\in L^\infty_c(\R^n)$,
with $C$ independent of $f$. For an account of the classical
weighted theory, see \cite[Chapter~7]{Grafakos}. For related
weighted classes and applications, see also \cite{WY2025,WZ2026}.

H\"older's inequality gives the monotonicity
$$
 [w]_{A_q}\le[w]_{A_p}\le[w]_{A_1},
 \qquad1<p\le q<\infty.
$$
Moreover, for each fixed cube $Q$, the normalized $L^r$ norms
of $w^{-1}$ increase to its essential supremum as $r\to\infty$.
Taking $r=\frac{1}{\varepsilon}$, we obtain
$$
 \lim_{\varepsilon\downarrow0}
 \left(\frac{1}{|Q|}\int_{Q}w^{-\frac{1}{\varepsilon}}\dx\right)^\varepsilon
 =\|w^{-1}\|_{L^\infty(Q)}
 =\frac{1}{\operatorname*{ess\,inf}_Qw}.
$$
Taking the supremum over $Q$ and using this monotonicity yields
$$
 A_1\subseteq\bigcap_{p>1}A_p,\qquad
 [w]_{A_1}
 =\lim_{\varepsilon\downarrow0}[w]_{A_{1+\varepsilon}}
 =\sup_{0<\varepsilon\le1}[w]_{A_{1+\varepsilon}}.
$$
The identity holds with values in $[1,\infty]$.
Membership in $\bigcap_{p>1}A_p$ requires finiteness of each
characteristic, whereas membership in $A_1$ requires a uniform
bound as $p\downarrow1$. We consider polynomial growth of
these characteristics. For $\beta\ge0$, we call $w$ a
Muckenhoupt $A_1$ growth weight of order $\beta$ and write
$w\in A_1^{(\beta)}$ if
\begin{equation}\label{eq:class}
 [w]_{A_1^{(\beta)}}:=
 \sup_{0<\varepsilon\le1}
       \varepsilon^\beta[w]_{A_{1+\varepsilon}}<\infty.
\end{equation}
In particular,
$
 A_1^{(0)}=A_1$ and $
 A_1\subsetneq A_1^{(\beta)}\subsetneq\bigcap_{p>1}A_p
$
for $\beta>0$.

The relation between polynomial growth of the characteristics
and logarithmic weak endpoint estimates goes back to Carbery,
Chang, and Garnett \cite{CCG}. Weighted weak-type inequalities in Orlicz spaces were studied in
\cite{KermanTorchinsky,Bagby,Pick,GP}; for iterated maximal operators, see \cite{Kita}.
%For $M$, we compare the best endpoint constant with the
%weight characteristic. We also treat $M_{L(\log L)^\alpha}$, singular
%integrals, and commutators.

For $t\ge0$, set
$$
 \Phi_\gamma(t)=t[\log(e+t)]^\gamma,\qquad\gamma\ge0.
$$
For real $\alpha\ge0$, define
\begin{equation}\label{eq:orliczdef}
 \begin{split}
 \|f\|_{L(\log L)^\alpha,Q}
 &=\inf\left\{a>0:
       \frac{1}{|Q|}\int_Q\Phi_\alpha\left(\frac{|f|}{a}\right)\dx\le1\right\},\\
 M_{L(\log L)^\alpha}f(x)
 &=\sup_{Q\ni x}\|f\|_{L(\log L)^\alpha,Q}.
 \end{split}
\end{equation}
These are Orlicz maximal operators. In the case $\alpha=1$,
one has the pointwise equivalence
$$
 M_{L\log L}f(x)\simeq_n M(Mf)(x).
$$
Carozza and Passarelli di Napoli \cite{CarozzaPassarelli}
studied compositions of Orlicz maximal operators. Their
results include the preceding equivalence and its analogues
for more general Young functions. Orlicz maximal operators
also occur in sufficient conditions for two-weight inequalities;
see P\'erez \cite{PerezMaximal}.
The standard properties used below are recalled in
Section~\ref{sec:proofs}. In the statements
below, the estimates are required for every
$f\in L^\infty_c(\R^n)$ and $\lambda>0$, with a constant
independent of $f$ and $\lambda$.

\begin{theorem}\label{thm:maximal}
Let $\alpha,\beta\ge0$. Then $w\in A_1^{(\beta)}$ if and
only if there is a constant $C<\infty$ such that
\begin{equation}\label{eq:orliczmaximal}
 w\bigl(\{M_{L(\log L)^\alpha}f>\lambda\}\bigr)
 \le C\int_{\R^n}\Phi_{\alpha+\beta}\left(\frac{|f(x)|}{\lambda}\right)w(x)\dx.
\end{equation}
If $C$ denotes the least constant in
\eqref{eq:orliczmaximal}, then
$$
 C\simeq_{n,\alpha,\beta}[w]_{A_1^{(\beta)}}.
$$
\end{theorem}

Since $\Phi_0(t)=t$, we have $M_{L(\log L)^0}=M$.
Taking $\alpha=0$ gives the following corollary.

\begin{corollary}\label{cor:maximal}
Let $\beta\ge0$. Then $w\in A_1^{(\beta)}$ if and only if
there is a constant $C<\infty$ such that
\begin{equation}\label{eq:maximal}
 w(\{Mf>\lambda\})\le C
                  \int_{\R^n}\Phi_\beta\left(\frac{|f(x)|}{\lambda}\right)w(x)\dx.
\end{equation}
The least constant satisfies
$C\simeq_{n,\beta}[w]_{A_1^{(\beta)}}$.
\end{corollary}

Let $T$ be an $L^2(\R^n)$-bounded Calder\'on--Zygmund
operator with
\begin{equation}\label{eq:Tdefinition}
 Tf(x)=\operatorname*{p.v.}\int_{\R^n}K(x,y)f(y)\,dy
 =\lim_{\delta\downarrow0}
       \int_{|x-y|>\delta}K(x,y)f(y)\,dy.
\end{equation}
The kernel satisfies, for some $0<\tau\le1$,
\begin{align}
 |K(x,y)|&\le c_K|x-y|^{-n},\label{eq:size}\\
 |K(x,y)-K(x',y)|+|K(y,x)-K(y,x')|
 &\le c_K\frac{|x-x'|^\tau}{|x-y|^{n+\tau}},
 \quad |x-x'|\le\frac{1}{2}|x-y|.\label{eq:smooth}
\end{align}
Its maximal truncation is
$$
 T^*f(x)=\sup_{\delta>0}
 \left|\int_{|x-y|>\delta}K(x,y)f(y)\,dy\right|.
$$
For the converse estimates, assume that there is $c_0>0$
such that, for every $y\in\R^n$ and $r>0$,
\begin{equation}\label{eq:nondegenerate}
 \min\{|K(x,y)|,|K(y,x)|\}\ge c_0r^{-n}
 \quad\text{for some }x\in\R^n\setminus B(y,r).
\end{equation}
The lower bound is required for both $K(x,y)$ and $K(y,x)$
at the same pair of points. This is stronger than the usual
nondegeneracy condition in
\cite[Definition~2.1.1]{Hytonen}, which only requires a lower
bound for $|K(x,y)|$.
\Needspace{8\baselineskip}
For convolution kernels $K(x,y)=k(x-y)$,
Stein's condition requires a unit vector $e$ and a constant
$c>0$ such that
$$
 |k(te)|\ge c|t|^{-n},\qquad
 t\in\R\setminus\{0\}.
$$
This implies \eqref{eq:nondegenerate}; see
\cite[Chapter~IV, Section~4.6]{Stein} and
\cite[Remark~2.1.2]{Hytonen}.
In particular, Hilbert and Riesz transforms satisfy
\eqref{eq:nondegenerate}.

\begin{theorem}\label{thm:CZ}
Let $\beta\ge0$, and let $T$ be a Calder\'on--Zygmund operator
satisfying \eqref{eq:nondegenerate}. Then $w\in A_1^{(\beta)}$
if and only if there is a constant $C<\infty$ such that
\begin{equation}\label{eq:CZ}
 w(\{T^*f>\lambda\})\le C
                   \int_{\R^n}\Phi_\beta\left(\frac{|f(x)|}{\lambda}\right)w(x)\dx.
\end{equation}
If $C$ is the least constant in \eqref{eq:CZ}, then
$$
 [w]_{A_1^{(\beta)}}\le C_{n,T,\beta}C^2,
 \qquad C\le C\bigl(n,T,\beta,[w]_{A_1^{(\beta)}}\bigr).
$$
The same assertions hold with $T^*f$ replaced by $|Tf|$.
\end{theorem}

For real $b\in\BMO(\R^n)$, write
$$
 \|b\|_{\BMO}=\sup_Q\left(\frac{1}{|Q|}\int_{Q}|b-b_Q|\dx\right),\qquad
 T_b^0=T,\qquad T_b^m=[b,T_b^{m-1}],\quad m\ge1.
$$
Here $[b,T]f=bTf-T(bf)$. Away from the support of $f$,
the kernel formula is
\begin{equation}\label{eq:commkernel}
 T_b^mf(x)=\int_{\R^n}(b(x)-b(y))^mK(x,y)f(y)\,dy.
\end{equation}
By the generalized H\"older inequality and the
John--Nirenberg inequality, we have
$$
 \frac{1}{|Q|}\int_Q|b(x)-b_Q|\,|f(x)|\dx
 \le C_n\|b\|_{\BMO}\|f\|_{L\log L,Q}.
$$
P\'erez \cite{Perez} established the $L\log L$ endpoint for
first-order commutators; see also
\cite{PerezTrujillo} for multilinear commutators. Higher commutators are compared
with $M^{m+1}$, equivalently $M_{L(\log L)^m}$;
see \cite{CGMP,LMRT}. Quantitative commutator bounds also
appear in \cite{CPP,LOR}.

\begin{theorem}\label{thm:commutator}
Let $\beta\ge0$, let $m\ge1$ be an integer, and let $T$ be
a Calder\'on--Zygmund operator satisfying
\eqref{eq:nondegenerate}. Then $w\in A_1^{(\beta)}$ if and
only if there is a constant $C<\infty$ such that
\begin{equation}\label{eq:commutator}
 w(\{|T_b^mf|>\lambda\})\le C
 \int_{\R^n}\Phi_{\beta+m}
       \left(\frac{\|b\|_{\BMO}^m|f|}{\lambda}\right)w\dx
\end{equation}
for every real $b\in\BMO$, with $C$ independent of $b$.
If $C$ is the least such constant, then
$$
 [w]_{A_1^{(\beta)}}\le C_{n,T,\beta,m}C^2,
 \qquad C\le C\bigl(n,T,\beta,m,[w]_{A_1^{(\beta)}}\bigr).
$$
\end{theorem}

The sufficient estimates in Theorems~\ref{thm:CZ}
and~\ref{thm:commutator} hold for every Calder\'on--Zygmund
operator satisfying \eqref{eq:size}--\eqref{eq:smooth}.
We do not claim optimal dependence on the characteristic in
these two estimates.

Section~\ref{sec:properties} contains the properties of
$A_1^{(\beta)}$ and the strict inclusion examples.
Section~\ref{sec:proofs} collects the auxiliary estimates
and proves the main theorems.
The notation $X\lesssim Y$ means $X\le cY$; subscripts
indicate the dependence of $c$. We write $X\simeq Y$ if
both inequalities hold.

\section{Properties of the weight class}\label{sec:properties}

We first record the inclusions between the classes
$A_1^{(\beta)}$ and give examples showing when they are strict.
We then consider products and powers, and give a local
characterization in terms of subsets of cubes.

The classes $A_1^{(\beta)}$ increase with $\beta$. The next
proposition records the inclusions and the corresponding
bounds for the $A_p$ characteristics.

\begin{proposition}\label{prop:basic}
Let $\beta\ge0$ and $w\in A_1^{(\beta)}$. Then $w\in A_p$
for every $p>1$, and
\begin{equation}\label{eq:Apcontrol}
 [w]_{A_p}\le
 \begin{cases}
 (p-1)^{-\beta}[w]_{A_1^{(\beta)}},&1<p\le2,\\[3pt]
 [w]_{A_1^{(\beta)}},&2<p<\infty.
 \end{cases}
\end{equation}
If $\gamma\ge\beta$, then $w\in A_1^{(\gamma)}$ and
\begin{equation}\label{eq:inclusions}
 [w]_{A_1^{(\gamma)}}\le[w]_{A_1^{(\beta)}}.
\end{equation}
Moreover, $A_1^{(0)}=A_1$, with equality of the characteristics.
\end{proposition}

\begin{proof}
The substitution $q=\frac{1}{\varepsilon}$ gives
\begin{equation}\label{eq:moment}
 \begin{split}
 [w]_{A_{1+\frac{1}{q}}}
 &=\sup_Q\left[\frac{1}{|Q|}\int_Q
       \left(\frac{w_Q}{w(x)}\right)^q\dx\right]^{\frac{1}{q}},
       \qquad q\ge1,\\
 [w]_{A_1^{(\beta)}}
 &=\sup_Q\sup_{q\ge1}\frac{1}{q^\beta}
       \left[\frac{1}{|Q|}\int_Q
       \left(\frac{w_Q}{w(x)}\right)^q\dx\right]^{\frac{1}{q}}.
 \end{split}
\end{equation}
For $\beta=0$, the normalized $L^q$ norms increase to the
essential supremum, so
\begin{equation}\label{eq:beta0}
 [w]_{A_1^{(0)}}=
 \sup_Q\left\|\frac{w_Q}{w}\right\|_{L^\infty(Q)}
 =[w]_{A_1}.
\end{equation}
The monotonicity in $\beta$ follows from
$\varepsilon^\gamma\le\varepsilon^\beta$ on $(0,1]$.
For $1<p\le2$, put $\varepsilon=p-1$ in \eqref{eq:class}.
For $p\ge2$, monotonicity gives
$[w]_{A_p}\le[w]_{A_2}\le[w]_{A_1^{(\beta)}}$.
\end{proof}

We next give examples for which the growth in
\eqref{eq:Apcontrol} is attained. For $\delta>0$, consider the weight
\begin{equation}\label{eq:exampleweight}
 w_\delta(x)=\bigl[1+\log^+\left(\frac{1}{|x_1|}\right)\bigr]^{-\delta},
 \qquad x_1\ne0,
\end{equation}
where $\log^+t=\max\{0,\log t\}$; its value on $\{x_1=0\}$
may be chosen arbitrarily in $(0,\infty)$. We claim that
\begin{equation}\label{eq:examplegrowth}
 [w_\delta]_{A_{1+\varepsilon}}\simeq_\delta
 \varepsilon^{-\delta},\qquad 0<\varepsilon\le1.
\end{equation}
Since $w_\delta$ depends only on $x_1$, its characteristic
over a cube equals the one-dimensional characteristic over
the first coordinate interval. It suffices to work on $\R$.
Set
$$
 L(t)=1+\log^+\left(\frac{1}{t}\right),\qquad t>0,\qquad q=\frac{1}{\varepsilon}\ge1.
$$
Then $w_\delta(t)=L(|t|)^{-\delta}$. For $r>0$ and $s\ge0$,
$$
 L(re^{-s})\le L(r)+s\le L(r)(1+s).
$$
The substitution $t=re^{-s}$ yields
\begin{equation}\label{eq:examplemoment}
 \frac{1}{r}\int_0^rL(t)^{\delta q}\,dt
 \le L(r)^{\delta q}
                   \int_0^\infty(1+s)^{\delta q}e^{-s}\,ds
 \le C_\delta^q q^{\delta q}L(r)^{\delta q}.
\end{equation}
The last bound follows from $(1+s)^{\delta q}\le2^{\delta q}(1+s^{\delta q})$ and
$$
 \int_0^\infty s^{\delta q}e^{-s}\,ds\le2\left(\frac{2\delta q}{e}\right)^{\delta q}.
$$
Taking a $q$th root gives a
constant depending only on $\delta$.

Let $I$ be an interval of length $\ell$ and
$d=\operatorname{dist}(I,0)$. If $d\ge\ell$, the interval lies on one side of zero and
$d\le|x|\le d+\ell$ for $x\in I$.
Since $d+\ell\le2d$,
$$
 L(d)\le L(d+\ell)+\log2
       \le(1+\log2)L(d+\ell).
$$
It follows that
$$
\frac{1}{|I|}\int_Iw_\delta(x)\dx
      \left(\frac{1}{|I|}\int_{I}w_\delta^{-q}\dx\right)^{\frac{1}{q}}
 \le\frac{\operatorname*{ess\,sup}_I w_\delta}{\operatorname*{ess\,inf}_I w_\delta}
 \le(1+\log2)^\delta.
$$
If $d<\ell$, then $I\subset(-2\ell,2\ell)$.
Monotonicity of $L$ gives $(w_\delta)_I\le L(2\ell)^{-\delta}$,
and \eqref{eq:examplemoment} gives
\begin{align*}
\frac{1}{|I|}\int_Iw_\delta(x)\dx
       \left(\frac{1}{|I|}\int_{I}w_\delta^{-q}\dx\right)^{\frac{1}{q}}
 &\le L(2\ell)^{-\delta}
       \left(\frac{2}{\ell}\int_0^{2\ell}L(t)^{\delta q}\,dt\right)^{\frac{1}{q}}\\
 &\le4^{\frac{1}{q}}C_\delta q^\delta\le4C_\delta q^\delta.
\end{align*}
Both cases prove the upper bound uniformly over all intervals.

For the lower bound take $I=(0,1)$ and write
$$
 a_\delta=\int_0^1[1+\log\left(\frac{1}{t}\right)]^{-\delta}\,dt>0.
$$
The same substitution gives
\begin{align*}
 \int_0^1[1+\log\left(\frac{1}{t}\right)]^{\delta q}\,dt
 &=\int_0^\infty(1+s)^{\delta q}e^{-s}\,ds\\
 &\ge\int_q^{q+1}(1+s)^{\delta q}e^{-s}\,ds
 \ge(1+q)^{\delta q}e^{-(q+1)}.
\end{align*}
Therefore
$$
 [w_\delta]_{A_{1+\frac{1}{q}}}
 \ge a_\delta(1+q)^\delta e^{-1-\frac{1}{q}}
 \ge e^{-2}a_\delta q^\delta.
$$
This proves \eqref{eq:examplegrowth}. The upper bound for
$1<p\le2$, followed by monotonicity for $p>2$, gives
$w_\delta\in\bigcap_{p>1}A_p$. On any interval containing
zero the essential infimum of $w_\delta$ is zero, so
$w_\delta\notin A_1$. Finally,
$$
 \sup_{0<\varepsilon\le1}
       \varepsilon^\beta[w_\delta]_{A_{1+\varepsilon}}
 \simeq_\delta
       \sup_{0<\varepsilon\le1}\varepsilon^{\beta-\delta},
$$
which is finite exactly when $\delta\le\beta$.
It follows that
$$
 w_\delta\in A_1^{(\beta)}\quad\Longleftrightarrow\quad
 \delta\le\beta.
$$
Taking $\delta=\beta$ and $\delta>\beta$, respectively, gives
\begin{equation}\label{eq:strictinclusions}
 A_1\subsetneq A_1^{(\beta)}
       \subsetneq\bigcap_{p>1}A_p,\qquad \beta>0.
\end{equation}
Also, $A_1^{(\beta)}\subsetneq A_1^{(\gamma)}$ whenever
$0\le\beta<\gamma$.

For $q\ge1$, the duality of Muckenhoupt characteristics gives
\begin{equation}\label{eq:dualgrowth}
 \begin{aligned}
 [w^{-q}]_{A_{1+q}}^{\frac{1}{q}}
 &=[w]_{A_{1+\frac{1}{q}}},\qquad q\ge1,\\
 [w]_{A_1^{(\beta)}}
 &=\sup_{q\ge1}q^{-\beta}[w^{-q}]_{A_{1+q}}^{\frac{1}{q}}.
 \end{aligned}
\end{equation}
Indeed, both sides of the first identity, raised to the
$q$th power, equal $\sup_Q\left(\frac{1}{|Q|}\int_{Q}w^{-q}\dx\right)(w_Q)^q$.

H\"older's inequality gives a product property for these
classes. Taking one of the weights to be constant then gives
the result for powers between zero and one. For powers greater
than one, we also use the reverse H\"older inequality; the
allowable range is determined by the weight characteristic.

\begin{proposition}\label{prop:powers}
Let $w_i\in A_1^{(\beta_i)}$, $i=0,1$, and $0\le\theta\le1$.
Then
\begin{equation}\label{eq:geometric}
 [w_0^{1-\theta}w_1^\theta]_
       {A_1^{((1-\theta)\beta_0+\theta\beta_1)}}
 \le[w_0]_{A_1^{(\beta_0)}}^{1-\theta}
       [w_1]_{A_1^{(\beta_1)}}^\theta.
\end{equation}
In particular, if $w\in A_1^{(\beta)}$ and $0<r\le1$, then
$w^r\in A_1^{(r\beta)}$ with characteristic at most
$[w]_{A_1^{(\beta)}}^r$. There is a dimensional constant
$c_n$ such that, if $A=[w]_{A_1^{(\beta)}}$ and
$1<r\le1+(c_nA)^{-1}$, then $w^r\in A_1^{(r\beta)}$ and
\begin{equation}\label{eq:powerup}
 [w^r]_{A_1^{(r\beta)}}\le2^rr^{r\beta}A^r.
\end{equation}
\end{proposition}

\begin{proof}
The cases $\theta=0,1$ are immediate. Otherwise H\"older's
inequality shows both local integrability of
$v=w_0^{1-\theta}w_1^\theta$ and
$$
 \begin{aligned}
 v_Q&\le(w_0)_Q^{1-\theta}(w_1)_Q^\theta,\\
 \frac{1}{|Q|}\int_Qv^{-\frac{1}{\varepsilon}}\dx
 &\le\left(\frac{1}{|Q|}\int_Qw_0^{-\frac{1}{\varepsilon}}\dx\right)^{1-\theta}
       \left(\frac{1}{|Q|}\int_Qw_1^{-\frac{1}{\varepsilon}}\dx\right)^\theta.
 \end{aligned}
$$
Multiplication and the definition of $A_1^{(\beta_i)}$ give
$$
 [v]_{A_{1+\varepsilon}}
 \le[w_0]_{A_1^{(\beta_0)}}^{1-\theta}
       [w_1]_{A_1^{(\beta_1)}}^\theta
       \varepsilon^{-((1-\theta)\beta_0+\theta\beta_1)},
$$
proving \eqref{eq:geometric}. Taking $w_0=1$, $\beta_0=0$,
$w_1=w$, and $\theta=r$ gives the assertion for $0<r\le1$.

For $r>1$, \eqref{eq:Apcontrol} gives $[w]_{A_2}\le A$.
The quantitative reverse H\"older inequality
\cite[Proposition~2.2 and Theorem~2.3]{HP} implies, after
fixing a dimensional $c_n$,
$$
 \left(\frac{1}{|Q|}\int_{Q}w^r\dx\right)^{\frac{1}{r}}\le2w_Q,
 \qquad1<r\le1+(c_nA)^{-1}.
$$
Thus $w^r$ is locally integrable. For $0<\varepsilon\le1$,
\begin{align*}
 [w^r]_{A_{1+\varepsilon}}
 &\le2^r\sup_Q(w_Q)^r
                   \left(\frac{1}{|Q|}\int_{Q}w^{-\frac{r}{\varepsilon}}\dx\right)^\varepsilon\\
 &=2^r[w]_{A_{1+\frac{\varepsilon}{r}}}^r
 \le2^r A^r\left(\frac{\varepsilon}{r}\right)^{-r\beta}.
\end{align*}
Multiplication by $\varepsilon^{r\beta}$ and a supremum prove
\eqref{eq:powerup}.
\end{proof}

The next proposition characterizes $A_1^{(\beta)}$ by a lower
bound for the weighted measure of subsets of a cube.

\begin{proposition}\label{lem:growth}
Let $\beta\ge0$. Then $w\in A_{1}^{(\beta)}$ implies
\begin{equation}\label{eq:density}
 \frac{w(E)}{w(Q)}\ge
 \frac{\rho}{e[w]_{A_1^{(\beta)}}[\log\left(\frac{e}{\rho}\right)]^\beta},
 \qquad\rho=\frac{|E|}{|Q|}>0,
\end{equation}
for every cube $Q$ and every measurable $E\subset Q$
with $|E|>0$. Conversely, if
\begin{equation}\label{eq:densityD}
 \frac{w(E)}{w(Q)}\ge
 \frac{\rho}{D[\log\left(\frac{e}{\rho}\right)]^\beta}
\end{equation}
holds for every cube $Q$ and every measurable $E\subset Q$
with $|E|>0$, with the same constant $D$, then
$w\in A_1^{(\beta)}$ and
$[w]_{A_1^{(\beta)}}\le C_\beta D$.
\end{proposition}

\begin{proof}
Let $E\subset Q$, with $\rho=\frac{|E|}{|Q|}>0$.
For $0<\varepsilon\le1$, H\"older's inequality gives
$$
 \begin{aligned}
 \rho
 &=\frac{1}{|Q|}\int_E
       w(x)^{\frac{1}{1+\varepsilon}}
       w(x)^{-\frac{1}{1+\varepsilon}}\dx\\
 &\le
 \left(\frac{w(E)}{|Q|}\right)^{\frac{1}{1+\varepsilon}}
 \left(\frac{1}{|Q|}\int_Qw(x)^{-\frac{1}{\varepsilon}}\dx
       \right)^{\frac{\varepsilon}{1+\varepsilon}}.
 \end{aligned}
$$
Raising both sides to the power $1+\varepsilon$ yields
\begin{equation}\label{eq:holderdensity}
 \rho^{1+\varepsilon}
 \le[w]_{A_{1+\varepsilon}}\frac{w(E)}{w(Q)}.
\end{equation}
Choose $\varepsilon=\frac{1}{\log\left(\frac{e}{\rho}\right)}$. Then
$$
 \rho^\varepsilon
 =\exp\left[-\frac{\log\left(\frac{1}{\rho}\right)}{1+\log\left(\frac{1}{\rho}\right)}\right]
 \ge e^{-1}.
$$
Substituting
$[w]_{A_{1+\varepsilon}}\le
[w]_{A_1^{(\beta)}}\varepsilon^{-\beta}$ into
\eqref{eq:holderdensity} proves \eqref{eq:density}.

Conversely, suppose that \eqref{eq:densityD} holds. If
$\beta=0$, fix $Q$ and $\eta>0$ and set
$$
 E=\left\{x\in Q:w(x)<\frac{w_Q}{D+\eta}\right\}.
$$
If $|E|>0$, then
$$
 \frac{|E|}{D|Q|}
 \le\frac{w(E)}{w(Q)}
 \le\frac{|E|}{(D+\eta)|Q|},
$$
which is impossible. Letting $\eta\downarrow0$, we obtain
$w(x)\ge\frac{w_Q}{D}$ for almost every $x\in Q$.
Thus $[w]_{A_1^{(0)}}=[w]_{A_1}\le D$.

Let $\beta>0$. For a fixed cube $Q$, write
$$
 E_s=\left\{x\in Q:\frac{w_Q}{w(x)}>s\right\},\qquad
 \rho_s=\frac{|E_s|}{|Q|}.
$$
Since $w(x)<\frac{w_Q}{s}$ on $E_s$,
$$
 \frac{w(E_s)}{w(Q)}\le\frac{\rho_s}{s}.
$$
If $\rho_s>0$, \eqref{eq:densityD} gives
$$
 \frac{\rho_s}{D[\log\left(\frac{e}{\rho_s}\right)]^\beta}
 \le\frac{\rho_s}{s}.
$$
Solving for $\rho_s$, and including the case $\rho_s=0$,
we obtain
\begin{equation}\label{eq:tail}
 \rho_s\le e\exp\left[-\left(\frac{s}{D}\right)^{\frac{1}{\beta}}\right],
 \qquad s>D.
\end{equation}
For $q\ge1$, integration of the distribution function and
the substitution $v=\left(\frac{s}{D}\right)^{\frac{1}{\beta}}$ give
\begin{align}
 \frac{1}{|Q|}\int_Q\left(\frac{w_Q}{w(x)}\right)^q\dx
 &=q\int_0^\infty s^{q-1}\rho_s\,ds\nonumber\\
 &\le D^q+eq\int_D^\infty s^{q-1}
       \exp\left[-\left(\frac{s}{D}\right)^{\frac{1}{\beta}}\right]\,ds
       \nonumber\\
 &\le D^q\left[1+e\beta q
                \int_0^\infty v^{\beta q-1}e^{-v}\,dv\right]
       \nonumber\\
 &=D^q\bigl[1+e\Gamma(\beta q+1)\bigr].
 \label{eq:momentcalculation}
\end{align}
The bound
$$
 \Gamma(r+1)=\int_0^\infty v^re^{-v}\,dv
 \le2\left(\frac{2r}{e}\right)^r,\qquad r>0,
$$
follows by maximizing $v^re^{-\frac{v}{2}}$ and integrating
$e^{-\frac{v}{2}}$. Hence
$$
 \left[\frac{1}{|Q|}\int_Q
       \left(\frac{w_Q}{w(x)}\right)^q\dx\right]^{\frac{1}{q}}
 \le C_\beta Dq^\beta.
$$
Taking the supremum over $Q$ and $q\ge1$ in
\eqref{eq:moment} proves
$[w]_{A_1^{(\beta)}}\le C_\beta D$.
\end{proof}

\section{Proofs of Theorems \ref{thm:maximal}, \ref{thm:CZ} and \ref{thm:commutator}}\label{sec:proofs}

We collect the elementary properties of $\Phi_\gamma$ and
the estimates for Orlicz maximal operators that will be used
in the proofs of the main theorems.

For $\gamma\ge0$, the function $\Phi_\gamma$ is convex,
$\Phi_\gamma(0)=0$, and $\Phi_\gamma(t)\ge t$. If $a\ge1$,
then
\begin{equation}\label{eq:dilation}
 \Phi_\gamma(at)\le a(1+\log a)^\gamma\Phi_\gamma(t),
 \qquad t\ge0.
\end{equation}
We shall also use the submultiplicative estimate
\begin{equation}\label{eq:submultiplicative}
 \Phi_\gamma(st)\le2^\gamma\Phi_\gamma(s)\Phi_\gamma(t),
 \qquad s,t\ge0.
\end{equation}
These follow from
$$
 \log(e+at)\le\log a+\log(e+t)
$$
and
$$
 \log(e+st)\le\log(e+s)+\log(e+t).
$$
Convexity gives the scaling inequality
\begin{equation}\label{eq:convexscale}
 \Phi_\gamma\left(\frac{t}{a}\right)\le a^{-1}\Phi_\gamma(t),
 \qquad a\ge1,\quad t\ge0.
\end{equation}
The Luxemburg averages are homogeneous and satisfy the
triangle inequality. For bounded $h$ on $Q$,
\begin{equation}\label{eq:normtest}
 \|h\|_{L(\log L)^\alpha,Q}>\lambda
 \quad\Longleftrightarrow\quad
 \frac{1}{|Q|}\int_Q\Phi_\alpha\left(\frac{|h(x)|}{\lambda}\right)\dx>1.
\end{equation}
For $h\not\equiv0$ on $Q$, the local $L(\log L)^\alpha$
norm has the following equivalent form; see
\cite{CarozzaPassarelli,Perez,CGMP}.
\begin{equation}\label{eq:averagenorm}
 \|h\|_{L(\log L)^\alpha,Q}
 \simeq_\alpha\frac{1}{|Q|}\int_Q
 |h(x)|\left[\log\left(e+
       \frac{|h(x)|}{|Q|^{-1}\int_Q|h(y)|\,dy}\right)\right]^\alpha\dx,
\end{equation}
and, for $E\subset Q$ with $\rho=\frac{|E|}{|Q|}>0$,
\begin{equation}\label{eq:indicator}
 \|\ind_E\|_{L(\log L)^\alpha,Q}
 =\frac{1}{\Phi_\alpha^{-1}(\rho^{-1})}
 \simeq_\alpha\rho[\log\left(\frac{e}{\rho}\right)]^\alpha.
\end{equation}

\begin{lemma}\label{lem:covering}
Let $\alpha,\gamma\ge0$. Suppose that for every cube $Q$,
every bounded $h\ge0$ supported on $Q$, and every $\lambda>0$,
\begin{equation}\label{eq:localtest}
 \|h\|_{L(\log L)^\alpha,Q}>\lambda
 \quad\Longrightarrow\quad
 w(Q)\le B\int_Q\Phi_\gamma\left(\frac{h(x)}{\lambda}\right)w(x)\dx.
\end{equation}
Then
\begin{equation}\label{eq:globaltest}
 w\bigl(\{M_{L(\log L)^\alpha}f>\lambda\}\bigr)
 \le C_{n,\gamma}B
       \int_{\R^n}\Phi_\gamma\left(\frac{|f(x)|}{\lambda}\right)w(x)\dx.
\end{equation}
Conversely, \eqref{eq:globaltest} with outer constant $B$
implies \eqref{eq:localtest} with the same constant.
\end{lemma}

\begin{proof}
Fix a compact $F\subset\{M_{L(\log L)^\alpha}f>\lambda\}$.
Each point of $F$ belongs to a cube on which the modular
average exceeds one. Enlarge these cubes slightly, preserving
this strict inequality, so that their interiors cover $F$.
From a finite subcover, select cubes in decreasing order of
side length and discard each cube meeting an already selected
cube. The selected cubes $Q_j$ are pairwise disjoint and satisfy
$$
 F\subset\bigcup_j3Q_j,\qquad
 \frac{1}{|Q_j|}\int_{Q_j}\Phi_\alpha\left(\frac{|f(x)|}{\lambda}\right)\dx>1.
$$
By \eqref{eq:convexscale},
$$
 \frac{1}{|3Q_j|}\int_{3Q_j}
 \Phi_\alpha\left(\frac{3^n|f(x)|\ind_{Q_j}(x)}{\lambda}\right)\dx
 \ge\frac{1}{|Q_j|}\int_{Q_j}\Phi_\alpha\left(\frac{|f(x)|}{\lambda}\right)\dx>1.
$$
Thus \eqref{eq:localtest}, applied on $3Q_j$, yields
$$
 \begin{aligned}
 w(F)&\le\sum_jw(3Q_j)
 \le B\sum_j\int_{Q_j}\Phi_\gamma\left(\frac{3^n|f(x)|}{\lambda}\right)w(x)\dx\\
 &\le3^n(1+n\log3)^\gamma B
       \int_{\R^n}\Phi_\gamma\left(\frac{|f(x)|}{\lambda}\right)w(x)\dx.
 \end{aligned}
$$
Inner regularity proves \eqref{eq:globaltest}.
Conversely,
$$
 \|h\|_{L(\log L)^\alpha,Q}>\lambda
 \quad\Longrightarrow\quad
 Q\subset\{M_{L(\log L)^\alpha}(h\ind_Q)>\lambda\},
$$
and the global estimate gives \eqref{eq:localtest}.
\end{proof}

For integers $k\ge0$, we use the standard equivalence
\begin{equation}\label{eq:iterate}
 M_{L(\log L)^k}f\simeq_{n,k}M^{k+1}f;
\end{equation}
see \cite[Section~6.1.2]{CGMP}. We also use the
Coifman--Rochberg estimate \cite{CR}:
\begin{equation}\label{eq:CR}
 \|\log Mh\|_{\BMO}\le C_n,
 \qquad0\le h\in L^1(\R^n),\quad h\not\equiv0.
\end{equation}

Write $A_\infty=\bigcup_{p>1}A_p$. The Coifman--Fefferman inequalities
\cite{CoifmanFefferman,Perez,CGMP,LMRT}, together with
\eqref{eq:iterate}, give, for $s>0$ and $v\in A_\infty$,
$$
 \begin{aligned}
 \int_{\R^n}(T^*f)^s v\dx
 &\le C\int_{\R^n}(Mf)^s v\dx,\\
 \int_{\R^n}|T_b^mf|^s v\dx
 &\le C\int_{\R^n}(M_{L(\log L)^m}f)^s v\dx,
 \qquad \|b\|_{\BMO}\le1.
 \end{aligned}
$$
For $\gamma\ge0$ and $w\in A_2$, the corresponding weak
modular estimate for $T^*$ is
\begin{equation}\label{eq:comparison}
 \sup_{t>0}\frac{w(\{T^*f>t\})}{\Phi_\gamma\left(\frac{1}{t}\right)}
 \le C\sup_{t>0}\frac{w(\{Mf>t\})}{\Phi_\gamma\left(\frac{1}{t}\right)}.
\end{equation}
For commutators, one has
\begin{equation}\label{eq:comparisoncomm}
 \sup_{t>0}\frac{w(\{|T_b^mf|>t\})}{\Phi_\gamma\left(\frac{1}{t}\right)}
 \le C\sup_{t>0}
       \frac{w(\{M_{L(\log L)^m}f>t\})}{\Phi_\gamma\left(\frac{1}{t}\right)},
\end{equation}
provided that $\|b\|_{\BMO}\le1$. Here $C$ depends on
$n,T,\gamma,m,[w]_{A_2}$, with $m$ omitted in
\eqref{eq:comparison}. These estimates follow from
\cite[Theorem~3.1, formula~(3.9)]{CGMP}.
Apply that theorem with
$\psi(t)=\frac{1}{\Phi_\gamma\left(\frac{1}{t}\right)}$.
Its scalar hypotheses follow from
$$
 \psi(2t)\le2(1+\log2)^\gamma\psi(t).
$$
Moreover, $\frac{\psi(t)^2}{t}$ is increasing, and hence
$$
 \int_0^t\frac{\psi(s)^2}{s}\,ds\simeq_\gamma\psi(t)^2.
$$
The integral is convex, so $\psi^2$ is quasi-convex.
Truncating the left-hand function in size and support gives
the required finiteness; monotone convergence removes the
truncation.

\vspace{0.1cm}

\begin{proof}[Proof of Theorem~\ref{thm:maximal}]
Suppose first that $w\in A_1^{(\beta)}$. Fix a cube $Q$
and a bounded function $h\ge0$ supported on $Q$ such that
$\|h\|_{L(\log L)^\alpha,Q}>\lambda$.
We verify the local estimate in Lemma~\ref{lem:covering}.
For $\beta=0$, \eqref{eq:beta0} gives
$\frac{w_Q}{w(x)}\le[w]_{A_1^{(0)}}$ almost everywhere on $Q$.
Consequently,
$$
 \begin{aligned}
 1&<\frac{1}{|Q|}\int_Q
       \Phi_\alpha\left(\frac{h(x)}{\lambda}\right)\dx\\
 &=\frac{1}{w(Q)}\int_Q\frac{w_Q}{w(x)}
       \Phi_\alpha\left(\frac{h(x)}{\lambda}\right)w(x)\dx\\
 &\le\frac{[w]_{A_1^{(0)}}}{w(Q)}
       \int_Q\Phi_\alpha\left(\frac{h(x)}{\lambda}\right)w(x)\dx.
 \end{aligned}
$$
This proves the required estimate in this case.

Let $\beta>0$. By Proposition~\ref{lem:growth},
\eqref{eq:densityD} holds with $D=e[w]_{A_1^{(\beta)}}$.
The argument giving \eqref{eq:tail} therefore yields
$$
 \frac{1}{|Q|}
 \left|\left\{x\in Q:\frac{w_Q}{w(x)}
          >e[w]_{A_1^{(\beta)}}t^\beta\right\}\right|
 \le
 \begin{cases}
 1,&0<t\le1,\\
 e^{1-t},&t>1.
 \end{cases}
$$
Integration of the distribution function gives
$$
 \begin{aligned}
 &\frac{1}{|Q|}\int_Q
   \exp\left[\frac{1}{4}
       \left(\frac{w_Q}{e[w]_{A_1^{(\beta)}}w(x)}\right)^{\frac{1}{\beta}}
       \right]\dx\\
 &\quad=1+\frac{1}{4}\int_0^\infty e^{\frac{t}{4}}
       \frac{\left|\left\{x\in Q:\frac{w_Q}{w(x)}
               >e[w]_{A_1^{(\beta)}}t^\beta\right\}\right|}{|Q|}\,dt\\
 &\quad\le1+\frac{1}{4}\int_0^1e^{\frac{t}{4}}\,dt
                  +\frac{e}{4}\int_1^\infty e^{-\frac{3t}{4}}\,dt
       =\frac{4}{3}e^{\frac{1}{4}}<2.
 \end{aligned}
$$
Thus, with $C_\beta=e4^\beta$,
\begin{equation}\label{eq:proofexp}
 \frac{1}{|Q|}\int_Q
 \exp\left[\left(
       \frac{w_Q}{C_\beta[w]_{A_1^{(\beta)}}w(x)}
       \right)^{\frac{1}{\beta}}\right]\dx\le2.
\end{equation}

We next verify the scalar inequality
\begin{equation}\label{eq:scalar}
 \Phi_\alpha(t)
 \le\frac{(\alpha+4)^\beta}{s}\Phi_{\alpha+\beta}(t)
       +\frac{1}{4}\exp\left(s^{\frac{1}{\beta}}\right),
 \qquad s>0,\quad t\ge0.
\end{equation}
If $\log(e+t)\ge\frac{s^{\frac{1}{\beta}}}{\alpha+4}$, then
$$
 \frac{(\alpha+4)^\beta}{s}\Phi_{\alpha+\beta}(t)
 =\frac{(\alpha+4)^\beta}{s}[\log(e+t)]^\beta\Phi_\alpha(t)
 \ge\Phi_\alpha(t).
$$
Otherwise, set $v=\frac{s^{\frac{1}{\beta}}}{\alpha+4}>1$.
Since $\log v\le v$,
$$
 \Phi_\alpha(t)\le v^\alpha e^v
 \le e^{(\alpha+1)v}
 =e^{-3v}e^{(\alpha+4)v}
 \le\frac{1}{4}\exp\left(s^{\frac{1}{\beta}}\right).
$$
This proves \eqref{eq:scalar}.

Apply \eqref{eq:scalar} with
$$
 s=\frac{w_Q}{C_\beta[w]_{A_1^{(\beta)}}w(x)},
 \qquad t=\frac{h(x)}{\lambda}.
$$
Averaging over $Q$ and using \eqref{eq:proofexp}, we obtain
$$
 \begin{aligned}
 1&<\frac{1}{|Q|}\int_Q
       \Phi_\alpha\left(\frac{h(x)}{\lambda}\right)\dx\\
 &\le\frac{C_\beta(\alpha+4)^\beta[w]_{A_1^{(\beta)}}}{w(Q)}
       \int_Q\Phi_{\alpha+\beta}
          \left(\frac{h(x)}{\lambda}\right)w(x)\dx+\frac{1}{2}.
 \end{aligned}
$$
It follows that
$$
 w(Q)\le2C_\beta(\alpha+4)^\beta[w]_{A_1^{(\beta)}}
       \int_Q\Phi_{\alpha+\beta}
          \left(\frac{h(x)}{\lambda}\right)w(x)\dx.
$$
Lemma~\ref{lem:covering} proves \eqref{eq:orliczmaximal}.

Conversely, suppose that \eqref{eq:orliczmaximal} holds with
constant $B$. Fix $E\subset Q$ with
$\rho=\frac{|E|}{|Q|}>0$, and set
$$
 r=\Phi_\alpha^{-1}\left(\frac{1}{\rho}\right),
 \qquad f=\ind_E,\qquad \lambda=\frac{1}{2r}.
$$
By \eqref{eq:indicator},
$$
 M_{L(\log L)^\alpha}\ind_E(x)
 \ge\|\ind_E\|_{L(\log L)^\alpha,Q}
 =\frac{1}{r}>\lambda,\qquad x\in Q.
$$
Therefore
$$
 \begin{aligned}
 w(Q)&\le B\Phi_{\alpha+\beta}(2r)w(E)\\
 &\le2(1+\log2)^{\alpha+\beta}B
       \Phi_{\alpha+\beta}(r)w(E)\\
 &=\frac{2(1+\log2)^{\alpha+\beta}B}{\rho}
       [\log(e+r)]^\beta w(E).
 \end{aligned}
$$
Here we used \eqref{eq:dilation} and
$\Phi_\alpha(r)=\frac{1}{\rho}$. Since $\Phi_\alpha(t)\ge t$,
we have $r\le\frac{1}{\rho}$, and hence
$$
 \log(e+r)\le\log\left(e+\frac{1}{\rho}\right)
 \le\log(e+1)\log\left(\frac{e}{\rho}\right).
$$
Consequently,
$$
 \frac{w(E)}{w(Q)}
 \ge\frac{\rho}{C_{\alpha,\beta}B
                   [\log\left(\frac{e}{\rho}\right)]^\beta}.
$$
The converse assertion in Proposition~\ref{lem:growth}
gives $w\in A_1^{(\beta)}$ and
$$
 [w]_{A_1^{(\beta)}}\le C_{\alpha,\beta}B.
$$
Combining the two estimates proves the comparison of the
least constant with $[w]_{A_1^{(\beta)}}$.
\end{proof}

\begin{proof}[Proof of Theorem~\ref{thm:CZ}]
Suppose that $w\in A_1^{(\beta)}$. By \eqref{eq:Apcontrol},
$[w]_{A_2}\le[w]_{A_1^{(\beta)}}$.
Using \eqref{eq:comparison}, Corollary~\ref{cor:maximal}, and
\eqref{eq:submultiplicative}, we obtain
$$
 \begin{aligned}
 \sup_{t>0}\frac{w(\{T^*f>t\})}{\Phi_\beta\left(\frac{1}{t}\right)}
 &\le C\sup_{t>0}\frac{w(\{Mf>t\})}{\Phi_\beta\left(\frac{1}{t}\right)}\\
 &\le C[w]_{A_1^{(\beta)}}
       \int_{\R^n}\Phi_\beta(|f(x)|)w(x)\dx.
 \end{aligned}
$$
Here $C$ may depend on $n,T,\beta$ and
$[w]_{A_1^{(\beta)}}$. Taking $t=1$ and replacing $f$ by
$\frac{f}{\lambda}$ proves sufficiency for $T^*$, and hence for $T$.

Conversely, assume the estimate for $|Tf|$ with constant $B$.
Fix a cube $Q$ with center $y_0$. For $R>4\sqrt n$,
\eqref{eq:nondegenerate} gives a point $x_0$ such that
$$
 \begin{aligned}
 R\ell(Q)&\le |x_0-y_0|
       \le\left(\frac{c_K}{c_0}\right)^{\frac{1}{n}}R\ell(Q)
        \end{aligned}
        $$
and
$$
 \begin{aligned}
 \min\{|K(x_0,y_0)|,|K(y_0,x_0)|\}
       &\ge\frac{c_0}{R^n|Q|}.
 \end{aligned}
$$
Let $Q'$ be the cube centered at $x_0$ with side length
$\ell(Q)$. Then $\operatorname{dist}(Q,Q')\ge\frac{R}{2}\ell(Q)$.
By \eqref{eq:smooth},
$$
 \begin{aligned}
 &\sup_{(x,y)\in Q'\times Q}
 \bigl(|K(x,y)-K(x_0,y_0)|
       +|K(y,x)-K(y_0,x_0)|\bigr)\\
 &\qquad\le\frac{C_{n,\tau}c_K}{R^{n+\tau}|Q|}
       \le\frac{c_0}{2R^n|Q|},
 \end{aligned}
$$
where $R$ is now fixed sufficiently large, depending only on
the kernel parameters. Set $c=\frac{c_0}{2R^n}$.
For a measurable $E\subset Q$ with $\rho=\frac{|E|}{|Q|}>0$,
write each kernel as its value at the centers plus the error.
The triangle inequality gives
$$
 \begin{aligned}
 |T\ind_{Q'}(y)|
 &\ge\left(|K(y_0,x_0)|-\frac{c}{|Q|}\right)|Q'|
       \ge c, &&y\in Q,\\
 |T\ind_E(x)|
 &\ge\left(|K(x_0,y_0)|-\frac{c}{|Q|}\right)|E|
       \ge c\rho, &&x\in Q'.
 \end{aligned}
$$
Apply the assumed estimate at levels $\frac{c}{2}$ and
$\frac{c\rho}{2}$, respectively. By \eqref{eq:dilation},
$$
 \begin{aligned}
 w(Q)
 &\le B\Phi_\beta\left(\frac{2}{c}\right)w(Q')\\
 &\le B^2\Phi_\beta\left(\frac{2}{c}\right)
             \Phi_\beta\left(\frac{2}{c\rho}\right)w(E)\\
 &\le C_{n,T,\beta}B^2
       \frac{\left[\log\left(\frac{e}{\rho}\right)\right]^\beta}{\rho}
       w(E).
 \end{aligned}
$$
Thus Proposition~\ref{lem:growth} yields
$$
 [w]_{A_1^{(\beta)}}\le C_{n,T,\beta}B^2.
$$
An estimate for $T^*$ implies the same estimate for $|T|$,
so necessity holds in both cases.
\end{proof}

\begin{proof}[Proof of Theorem~\ref{thm:commutator}]
Suppose that $w\in A_1^{(\beta)}$ and $\|b\|_{\BMO}\le1$.
By \eqref{eq:comparisoncomm}, Theorem~\ref{thm:maximal}, and
\eqref{eq:submultiplicative},
$$
 \begin{aligned}
 \sup_{t>0}\frac{w(\{|T_b^mf|>t\})}{\Phi_{\beta+m}\left(\frac{1}{t}\right)}
 &\le C\sup_{t>0}
       \frac{w(\{M_{L(\log L)^m}f>t\})}{\Phi_{\beta+m}\left(\frac{1}{t}\right)}\\
 &\le C[w]_{A_1^{(\beta)}}
       \int_{\R^n}\Phi_{\beta+m}(|f(x)|)w(x)\dx.
 \end{aligned}
$$
The constant may depend on $n,T,\beta,m$ and
$[w]_{A_1^{(\beta)}}$, but is independent of $b$.
Taking $t=1$ and replacing $f$ by $\frac{f}{\lambda}$ gives the
estimate for $\|b\|_{\BMO}\le1$. For $\|b\|_{\BMO}>0$, use
$$
 T_b^mf=\|b\|_{\BMO}^mT_{\frac{b}{\|b\|_{\BMO}}}^mf.
$$

Conversely, assume \eqref{eq:commutator} with constant $B$.
For $\|b\|_{\BMO}\le1$, it implies
\begin{equation}\label{eq:commassumed}
 w(\{|T_b^mf|>\lambda\})\le B
       \int_{\R^n}\Phi_{\beta+m}\left(\frac{|f(x)|}{\lambda}\right)w(x)\dx.
\end{equation}
Fix a cube $Q$ with center $y_0$. Choose $R>4\sqrt n$
sufficiently large and apply \eqref{eq:nondegenerate} with
$r=R\ell(Q)$. Let $Q'$ have center $x_0$ and side length
$\ell(Q)$, where $x_0$ is the resulting point. The size and
smoothness estimates give
$$
 \begin{gathered}
 \frac{R}{2}\ell(Q)\le\operatorname{dist}(Q,Q')
       \le C_{n,T}R\ell(Q),\\
 \min\{|K(x_0,y_0)|,|K(y_0,x_0)|\}
       \ge\frac{c_0}{R^n|Q|},\\
 \sup_{(x,y)\in Q'\times Q}
 \bigl(|K(x,y)-K(x_0,y_0)|
       +|K(y,x)-K(y_0,x_0)|\bigr)
       \le\frac{c_0}{2R^n|Q|}.
 \end{gathered}
$$
Put $c=\frac{c_0}{2R^n}$. Since $\|\ind_{Q'}\|_{\BMO}\le1$,
for $y\in Q$ we have
$$
 \begin{aligned}
 |T_{\ind_{Q'}}^m\ind_{Q'}(y)|
 &=|T\ind_{Q'}(y)|\\
 &\ge |K(y_0,x_0)|\,|Q'|
       -\int_{Q'}|K(y,x)-K(y_0,x_0)|\,dx\ge c.
 \end{aligned}
$$
Applying \eqref{eq:commassumed} at level $\frac{c}{2}$ gives
\begin{equation}\label{eq:companioncomm}
 w(Q)\le B\Phi_{\beta+m}\left(\frac{2}{c}\right)w(Q')
       \le C_{T,\beta,m}B w(Q').
\end{equation}

Let $h\ge0$ be bounded, nonzero, and supported on $Q$.
Write $h_Q=|Q|^{-1}\int_Qh(y)\,dy$ and
$N=\|h\|_{L(\log L)^m,Q}$. Define
$$
 b=\frac{1}{C_n}\log M\bigl[\left(e+\frac{h}{h_Q}\right)\ind_Q\bigr].
$$
By \eqref{eq:CR}, $C_n$ can be fixed so that
$\|b\|_{\BMO}\le1$. Lebesgue differentiation and the
separation of $Q,Q'$ give
$$
 \begin{aligned}
 C_nb(y)&\ge\log\left(e+\frac{h(y)}{h_Q}\right), &&\text{a.e. }y\in Q,\\
 C_nb(x)&\le C_{n,T}, &&x\in Q'.
 \end{aligned}
$$
For the second line, use
$\int_Q\left(e+\frac{h}{h_Q}\right)=(e+1)|Q|$ and the lower bound for the
side length of any cube meeting both $Q$ and $Q'$.
Fix $L>2\max\{1,C_{n,T}\}$ and set
$$
 E=\{y\in Q:\log\left(e+\frac{h(y)}{h_Q}\right)>L\}.
$$
Then
\begin{equation}\label{eq:localcommdominance}
 b(y)-b(x)\ge\frac{1}{2C_n}\log\left(e+\frac{h(y)}{h_Q}\right),
 \qquad (x,y)\in Q'\times E.
\end{equation}
If
$$
 \int_Eh(y)[\log\left(e+\frac{h(y)}{h_Q}\right)]^m\,dy
 \ge\frac{1}{2}\int_Qh(y)[\log\left(e+\frac{h(y)}{h_Q}\right)]^m\,dy,
$$
then $b(y)-b(x)>0$ for $(x,y)\in Q'\times E$.
By \eqref{eq:commkernel}, \eqref{eq:averagenorm}, and
\eqref{eq:localcommdominance}, for $x\in Q'$ we obtain
$$
 \begin{aligned}
 |T_b^m(h\ind_E)(x)|
 &\ge |K(x_0,y_0)|\int_E(b(y)-b(x))^m h(y)\,dy\\
 &\quad-\int_E(b(y)-b(x))^m
               |K(x,y)-K(x_0,y_0)|h(y)\,dy\\
 &\ge\frac{c}{|Q|}\int_E(b(y)-b(x))^m h(y)\,dy\\
 &\ge\frac{c}{(2C_n)^m|Q|}
       \int_Eh(y)\left[\log\left(e+\frac{h(y)}{h_Q}\right)\right]^m\,dy
 \ge c_{n,T,m}N.
 \end{aligned}
$$
Applying \eqref{eq:commassumed} at level $\frac{c_{n,T,m}N}{2}$
and using \eqref{eq:dilation}, we obtain
\begin{equation}\label{eq:commQprime}
 w(Q')\le C_{n,T,\beta,m}B
       \int_Q\Phi_{\beta+m}\left(\frac{h(y)}{N}\right)w(y)\,dy.
\end{equation}
In the remaining case,
$$
 N\lesssim_m\frac{2}{|Q|}
       \int_{Q\setminus E}h(y)[\log\left(e+\frac{h(y)}{h_Q}\right)]^m\,dy
 \le2L^mh_Q.
$$
Hence, for $x\in Q'$,
$$
 \begin{aligned}
 |T_{\ind_Q}^mh(x)|=|Th(x)|
 &\ge |K(x_0,y_0)|\int_Qh(y)\,dy
       -\int_Q|K(x,y)-K(x_0,y_0)|h(y)\,dy\\
 &\ge ch_Q\ge c_{n,T,m}N.
 \end{aligned}
$$
The same application of \eqref{eq:commassumed}, now with
$b=\ind_Q$, proves \eqref{eq:commQprime}.
Combining it with \eqref{eq:companioncomm} gives
$$
 N>\lambda
 \quad\Longrightarrow\quad
 w(Q)\le C_{n,T,\beta,m}B^2
       \int_Q\Phi_{\beta+m}\left(\frac{h(y)}{\lambda}\right)w(y)\,dy.
$$
Lemma~\ref{lem:covering} and Theorem~\ref{thm:maximal}
therefore yield
$$
 w(\{M_{L(\log L)^m}f>\lambda\})
 \le C_{n,T,\beta,m}B^2
       \int_{\R^n}\Phi_{\beta+m}\left(\frac{|f(x)|}{\lambda}\right)w(x)\dx,
$$
and hence
$[w]_{A_1^{(\beta)}}\le C_{n,T,\beta,m}B^2$.
\end{proof}

\section*{Acknowledgements}
This work was supported by the National Natural Science Foundation
of China (No.~12101010), the Natural Science Foundation of Anhui
Province (No.~2108085QA19), and the Key Scientific Project of
Higher Education Institutions in Anhui Province
(No.~2023AH050145).

\medskip
\noindent\textbf{Conflict of interest.}
The author declares that there is no conflict of interest in
relation to this article.

\medskip
\noindent\textbf{Data availability.}
Data sharing is not applicable to this article, as no datasets
were generated or analysed during the current study.

\medskip
\noindent\textbf{Use of AI tools.}
The author developed the main ideas and determined the overall
direction of this work. GPT-6 Astra was used to assist with the
geometric argument and the direct proof of the necessary
condition in Theorem~\ref{thm:CZ}, as well as
with checking scaling calculations and revising the exposition.
The author takes full responsibility for the mathematical
content of the paper.

\bigskip
\noindent\textsc{School of Mathematics and Statistics,
Anhui Normal University, Wuhu 241002, China.}

\smallskip
\noindent\textit{Email address:}
\href{mailto:Wangdh1990@126.com}{\texttt{Wangdh1990@126.com}}

\end{document}